\documentclass[12pt,a4paper,reqno]{amsart}
\usepackage{amsmath}
\allowdisplaybreaks
\usepackage{amsfonts}
\usepackage{amssymb,amsthm,amsfonts,amsthm,latexsym,enumerate,url,cases}
\numberwithin{equation}{section}
\usepackage{relsize}
\usepackage{mathrsfs}
\usepackage{hyperref}
\hypersetup{colorlinks=true,citecolor=blue,linkcolor=blue,urlcolor=blue}
\usepackage{extarrows}
\usepackage{tikz,xcolor,hyperref}

\definecolor{lime}{HTML}{A6CE39}
\DeclareRobustCommand{\orcidicon}{
    \begin{tikzpicture}
    \draw[lime, fill=lime] (0,0)
    circle [radius=0.16]
    node[white] {{\fontfamily{qag}\selectfont \tiny ID}};
    \draw[white, fill=white] (-0.0625,0.095)
    circle [radius=0.007];
    \end{tikzpicture}
    \hspace{-2mm}
}
\foreach \x in {A, ..., Z}{\expandafter\xdef\csname orcid\x\endcsname{\noexpand\href{https://orcid.org/\csname orcidauthor\x\endcsname}
            {\noexpand\orcidicon}}
}

\newtheorem{theorem*}{Theorem}
\newtheorem{lemma*}{Lemma}
\theoremstyle{plain}

\newtheorem{theorem}{Theorem}
\newtheorem{lemma}[theorem]{Lemma}

\newtheorem{problem}{Problem}
\theoremstyle{definition}
\newtheorem*{acknowledgment}{Acknowledgments}
\newtheorem{remark}{Remark}


\begin{document}

\title
[{Correct order on some certain weighted representation functions}] {Correct order on some certain weighted representation functions}

\author
[S.--Q. Chen \quad \& \quad Y. Ding* \quad \& \quad X. L\"u \quad \& \quad Y. Zhang] {Shi--Qiang Chen \quad \& \quad Yuchen Ding* \quad \& \quad Xiaodong L\"u \quad \& \quad Yuhan Zhang}

\address{(Shi--Qiang Chen) School of Mathematics and Statistics,  Anhui Normal University, Wuhu 241002, People's Republic of China}
\email{csq20180327@163.com}
\address{(Yuchen Ding) School of Mathematical Sciences,  Yangzhou University, Yangzhou 225002, People's Republic of China}
\email{ycding@yzu.edu.cn}
\address{(Xiaodong L\"u) School of Mathematical Sciences,  Yangzhou University, Yangzhou 225002, People's Republic of China}
\email{xdlv@yzu.edu.cn}
\address{(Yuhan Zhang) School of Mathematical Sciences,  Yangzhou University, Yangzhou 225002, People's Republic of China} \email{Qiaoyuan0804@hotmail.com}
\thanks{*Corresponding author}


\begin{abstract}
Let $\mathbb{N}$ be the set of all nonnegative integers. For any positive integer $k$ and any subset $A$ of nonnegative integers, let $r_{1,k}(A,n)$ be the number of solutions $(a_1,a_2)$ to the equation $n=a_1+ka_2$. In 2016, Qu proved that
$$\liminf_{n\rightarrow\infty}r_{1,k}(A,n)=\infty$$
providing that $r_{1,k}(A,n)=r_{1,k}(\mathbb{N}\setminus A,n)$ for all sufficiently large integers, which answered affirmatively a 2012 problem of Yang and Chen. In a very recent article, another Chen (the first named author) slightly improved Qu's result and obtained that
$$\liminf_{n\rightarrow\infty}\frac{r_{1,k}(A,n)}{\log n}>0.$$
In this note, we further improve the lower bound on $r_{1,k}(A,n)$ by showing that
$$\liminf_{n\rightarrow\infty}\frac{r_{1,k}(A,n)}{n}>0.$$
Our bound reflects the correct order of magnitude of the representation function $r_{1,k}(A,n)$ under the above restrictions due to the trivial fact that $r_{1,k}(A,n)\le n/k.$
\end{abstract}
\maketitle

\baselineskip 18pt

\section{Introduction}
Let $\mathbb{N}$ be the set of all nonnegative integers and $A$ a subset of $\mathbb{N}$. For any nonnegative integer $n$, let $R_1(A,n);R_2(A,n)$ and $R_3(A,n)$ be the number of solutions $(a,a')$ to the equations $n=a+a'$ with $a,a'\in A$; $a,a'\in A,~a<a'$ and $a,a'\in A,~a\le a'$, respectively. For backgrounds on these representation functions $R_i(A,n)$, $i=1,2,3$, one can refer to an early survey article of  S\'{a}rk\"{o}zy and S\'{o}s \cite{Sar-S}.
Following S\'{a}rk\"{o}zy's question, Dombi \cite{Dombi}, Chen and Wang \cite{Chen-Wang}, Lev \cite{Lev}, S\'{a}ndor \cite{Sandor}, Tang \cite{Tang}, Chen and Tang \cite{Chen-Tang} and Chen \cite{Chen} investigated various properties on values of the representation functions $R_i(A,n)$ and $R_i(\mathbb{N}\setminus A,n),i=1,2,3$.

In an interesting paper, Yang and Chen \cite{Yang-Chen} introduced the following weighted representation function
$$r_{k_1,k_2}(A,n)=\#\left\{(a_1,a_2)\in A^2: n=k_1a_1+k_2a_2\right\},$$
where $A$ is a subset of $\mathbb{N}$ and $k_1,k_2$ are two positive integers. They determined all pairs $(k_1, k_2)$ of positive integers for which there exists a set $A\subseteq \mathbb{N}$ such that
$$r_{k_1,k_2}(A,n)=r_{k_1,k_2}\left(\mathbb{N}\setminus A,n\right)$$
for all sufficiently large integers, which would reduce to partial answers to the original question of S\'{a}rk\"{o}zy mentioned above for $k_1=k_2=1$ on $R_1(A,n)$. For $1\le k_1<k_2$ with $(k_1,k_2)=1$, if there exists a set $A\subseteq \mathbb{N}$ such that
$$r_{k_1,k_2}(A,n)=r_{k_1,k_2}\left(\mathbb{N}\setminus A,n\right)$$
for all sufficiently large integers, then Yang and Chen proved that $k_1=1$. So the studies of the weighted representation function
$r_{1,k}(A,n)$ would be of particular interest.

For a positive integer $k\ge 2$, let $\Psi_k$ be the set of all $A\subseteq \mathbb{N}$ such that
$$r_{1,k}(A,n)=r_{1,k}\left(\mathbb{N}\setminus A,n\right)$$
for sufficiently large integers $n$. A result of Yang \cite{Yang} states that  if $k, \ell$ are
multiplicatively independent (equivalently, $\log k/\log \ell$ is irrational), then
$\Psi_k\cap\Psi_\ell=\emptyset$. Qu \cite{Qu} then gave a complete criteria for which
$\Psi_k\cap\Psi_\ell=\emptyset$.  It turns out to be that
$\Psi_k\cap\Psi_\ell\neq\emptyset$ if and only if $\log k/\log \ell=a/b$ for some odd positive integers $a$ and $b$, which disproved a conjecture of Yang \cite{Yang}. For related result, see also the article of Li and Ma \cite{LiMa}.
In \cite{Yang-Chen}, Yang and Chen posed the following problem:
\begin{problem}\label{problem1}
For any set $A\in \Psi_k$, is it true that $r_{1,k}(A, n)\ge1$
for all sufficiently large integers $n$? Is it true that $r_{1,k}(A, n)\rightarrow\infty$ as $n \rightarrow\infty$?
\end{problem}
Problem \ref{problem1} was later answered affirmatively by Qu \cite{Qu}. Very recently, Chen \cite{Chen1} improved Qu's result by showing that
$$\liminf_{n\rightarrow\infty}\frac{r_{1,k}(A,n)}{\log n}>0$$
for any $A\in \Psi_k$. In this note, we give the following very much stronger bound.

\begin{theorem}\label{thm1}
Let $k\ge 2$ be a given integer. For any set $A\in \Psi_k$, we have
$$\liminf_{n\rightarrow\infty}\frac{r_{1,k}(A,n)}{n}>0.$$
\end{theorem}
\begin{remark}
It can be seen that our new bound is sharp on the order of magnitude of $r_{1,k}(A,n)$ in the sense that
$$\limsup_{n\rightarrow\infty}\frac{r_{1,k}(A,n)}{n}\le 1/k.$$
Perhaps, it should be pointed out that the original argument of Qu \cite{Qu} with necessary adjustments can also lead to the bound given by Chen \cite{Chen1}.
It is also worth mentioning that our argument here leading to the sharp bound above in Theorem \ref{thm1} is different and simplified comparing with the somewhat complicated ones taken by Qu and Chen.
\end{remark}

\section{Proofs}
Following Qu \cite{Qu}, we may write $A$ as the following union of the `blocks'
$$A=\bigcup_{i=0}^{\infty}\left[t_{2i},t_{2i+1}\right),$$
where  $0\le t_0<t_1<t_2<\cdot\cdot\cdot$ is an increasing sequence of integers.
The proof of our theorem is based on the following lemma of Qu \cite[Lemma 2.1]{Qu}.

\begin{lemma}\label{lemma1}
Let $k\ge 2$ be a given integer. For any $A\in \Psi_k$ with $$A=\bigcup_{i=0}^{\infty}\left[t_{2i},t_{2i+1}\right),$$ there exist an odd positive integer $a$ and a nonnegative integer $i_0$ such that $t_{i+a}= kt_i$ for all $i\ge  i_0$.
\end{lemma}

 \begin{proof}[Proof of theorem \ref{thm1}]
By Lemma \ref{lemma1} for $A\in \Psi_k$ with $$A=\bigcup_{i=0}^{\infty}\left[t_{2i},t_{2i+1}\right),$$ there exists an odd positive integer $a$ such that $t_{i+a}= kt_i$ for all $i\ge  i_0$. Without loss of generality, we can assume that $i_0=0$, otherwise one can consider $\widetilde{A}=A\setminus[0,t_{i_0})$ instead of $A$ (this can be seen from the proofs below).
 Let $T=4(t_{a+2}-t_0).$
Then there exists some odd integer $g\in \mathbb{N}$ such that $k^g>T$.

From now on, let $n$ be a sufficiently large number. It is clear that there are nonnegative integers $m$ and $r$ with $0\le r<(k^g+1)$ such that
$$n=(k^g+1)m+r.$$
We can assume that $m\in \left[k^st_\ell, k^st_{\ell+1}\right)$ for two nonnegative integers $s$ and $\ell$ with $$0\le \ell\le a-1.$$ 
Recall that the integer $n$ is assumed to be sufficiently large, it follows that both $m$ and $s$ are sufficiently large.
We will prove that
$$r_{1,k}(A,n)\ge \frac{n}{k^5t_a(k^g+2)}-(k^g+1),$$
from which it follows clearly that
$$\liminf_{n\rightarrow\infty}\frac{r_{1,k}(A,n)}{n}>0.$$
Since $r_{1,k}(A,n)=r_{1,k}(\mathbb{N}\setminus A,n)$ for all sufficiently large $n$, we can also assume, without loss of generality, that $$\left[k^st_\ell, k^st_{\ell+1}\right)\subseteq A.$$
Since $g$ is an odd integer, we make the observation that each of the following three intervals
$$\left[k^st_{\ell-2}, k^st_{\ell-1}\right), \quad \left[k^st_{\ell+2}, k^st_{\ell+3}\right) \quad \text{and} \quad \left[k^{s+g-1}t_\ell, k^{s+g-1}t_{\ell+1}\right)$$ contains in $A$ as well, which is crucial in the following arguments. Before the continuation of the proof, we make the following notice that for brevity we write, for example, $$k^3t_2, \quad k^4t_{2-a} \quad \text{and} \quad k^2t_{2+a}$$
as the same number at different occasions.
The proofs are divided into three cases:

{\bf Case I.} $k^st_\ell+k^{s-4}\le m<k^st_{\ell+1}-k^{s-4}.$
Noting that for any $0\le q< k^{s-5}-r$, we have
$$n=(m+kq+r)+k\left(k^{g-1}m-q\right),$$
where both $m+kq+r$ and $k^{g-1}m-q$ belong to $A$ since
$$m+kq+r\in \left[k^st_\ell, k^st_{\ell+1}\right) \quad \text{and} \quad k^{g-1}m-q\in \left[k^{s+g-1}t_\ell, k^{s+g-1}t_{\ell+1}\right).$$
Thus, we deduce that
\begin{align*}
r_{1,k}(A,n)&\ge k^{s-5}-r\\
&>\frac{m}{k^5t_{\ell+1}}-(k^g+1)\\
&\ge \frac{n-r}{k^5t_a(k^g+1)}-(k^g+1)\\
&\ge \frac{n}{k^5t_a(k^g+2)}-(k^g+1).
\end{align*}

{\bf Case II.} $k^st_\ell\le m<k^st_\ell+k^{s-4}$. For any $$k^{s-1}(t_{\ell}-t_{\ell-1})+k^{s-5}+r< q\le k^{s-1}(t_{\ell}-t_{\ell-2}),$$ it can be seen that $$m-kq+r\in \left[k^st_{\ell-2}, k^st_{\ell-1}\right)\subseteq A$$ 
and
$$k^{g-1}m+q\in \left[k^{s+g-1}t_\ell, k^{s+g-1}t_{\ell+1}\right)\subseteq A,$$
where the latter inclusion relation comes from the observation that $\left[k^{s+g-1}t_\ell, k^{s+g-1}t_{\ell+1}\right)$ contains in $A$ made previously and the
facts that
$$k^{g-1}m+q<k^{g-1}\left(k^st_\ell+k^{s-4}\right)+k^{s-1}(t_{\ell}-t_{\ell-2})\le k^{s+g-1}t_{\ell+1}$$
since $k^g>T\ge 2(t_{\ell}-t_{\ell-2}).$ Note that
\begin{equation*}
n=(m-kq+r)+k\left(k^{g-1}m+q\right)
\end{equation*}
for all these $q$, from which we conclude that
\begin{align*}
r_{1,k}(A,n)&\ge k^{s-1}(t_{\ell}-t_{\ell-2})-k^{s-1}(t_{\ell}-t_{\ell-1})-k^{s-5}-r\\
&\ge \frac{1}{2}k^{s-1}-r\\
&\ge\frac{n-r}{2kt_a(k^g+1)}-(k^g+1)\\
&\ge \frac{n}{2kt_a(k^g+2)}-(k^g+1).
\end{align*}

{\bf Case III.} $k^st_{\ell+1}-k^{s-4}\le m<k^st_{\ell+1}$. It can be verified directly that
$$m+kq+r\in \left[k^st_{\ell+2}, k^st_{\ell+3}\right)\subseteq A$$
and
$$k^{g-1}m-q\in \left[k^{s+g-1}t_\ell, k^{s+g-1}t_{\ell+1}\right)\subseteq A,$$
for any $$k^{s-1}(t_{\ell+2}-t_{\ell+1})+k^{s-5}\le q\le k^{s-1}(t_{\ell+3}-t_{\ell+1})-r$$
via similar arguments in {\bf Case II.} In fact,
$$k^{g-1}m-q\ge k^{g-1}\left(k^st_{\ell+1}-k^{s-4}\right)-k^{s-1}(t_{\ell+3}-t_{\ell+1})\ge k^{s+g-1}t_\ell$$
since $k^g>T\ge 2(t_{\ell+3}-t_{\ell+1})$. Note that
\begin{equation*}
n=(m+kq+r)+k\left(k^{g-1}m-q\right)
\end{equation*}
for all these $q$, from which we conclude that
\begin{align*}
r_{1,k}(A,n)&\ge k^{s-1}(t_{\ell+3}-t_{\ell+1})-k^{s-1}(t_{\ell+2}-t_{\ell+1})-k^{s-5}-r\\
&\ge \frac{1}{2}k^{s-1}-r\\
&\ge\frac{n-r}{2kt_a(k^g+1)}-(k^g+1)\\
&\ge \frac{n}{2kt_a(k^g+2)}-(k^g+1).
\end{align*}

This completes the proof of Theorem \ref{thm1}.
\end{proof}

\begin{acknowledgment}

The authors would like to thank the anonymous referee for his helpful comments.

The second author is supported by National
Natural Science Foundation of China under Grant No. 12201544, Natural Science Foundation of Jiangsu Province, China, Grant No. BK20210784, China Postdoctoral Science Foundation, Grant No. 2022M710121, the foundations of the projects "Jiangsu Provincial Double--Innovation Doctor Program'', Grant No. JSSCBS20211023 and "Golden  Phoenix of the Green City--Yang Zhou'' to excellent PhD, Grant No. YZLYJF2020PHD051.

The third author is supported by National
Natural Science Foundation of China under Grant No. 12101538.
\end{acknowledgment}

\end{document}